\documentclass[reqno]{amsart}
\usepackage{CJK}
\usepackage{hyperref}
\usepackage{cite}
\usepackage{amsmath}
\usepackage{mathrsfs}
\usepackage{bm}
\usepackage{graphicx}
\usepackage{epsfig}
\usepackage{epstopdf}
\usepackage{amssymb}
\usepackage{dutchcal}

\numberwithin{equation}{section}

\newtheorem{theorem}{Theorem}[section]
\newtheorem{lemma}[theorem]{Lemma}
\newtheorem{definition}[theorem]{Definition}

\newtheorem{proposition}[theorem]{Proposition}
\newtheorem{remark}[theorem]{Remark}

\allowdisplaybreaks

\makeatletter
\@namedef{subjclassname@2020}{%
	\textup{2020} Mathematics Subject Classification}
\makeatother

\newcommand{ \mint }{ {\int\hspace{-0.38cm}-}}

\begin{document}
	
\title[\hfil Weak Harnack estimates for nonlocal Trudinger equation] {Weak Harnack estimates for a doubly nonlinear nonlocal $p$-Laplace equation}

\author[B. Shang, C. Zhang \hfil \hfilneg]
{Bin Shang and Chao Zhang$^*$}

\thanks{$^*$Corresponding author.}

\address{Bin Shang \hfill\break
	School of Mathematics, Harbin Institute of Technology,
	Harbin 150001, P.R. China} \email{shangbin0521@163.com}

\address{Chao Zhang\hfill\break
	School of Mathematics and Institute for Advanced Study in Mathematics, Harbin Institute of Technology,
	Harbin 150001, P.R. China} \email{czhangmath@hit.edu.cn}


\maketitle

\begin{abstract}
We establish a new type of weak Harnack estimates with optimal parabolic tail for the weak supersolutions to a doubly nonlinear nonlocal $p$-Laplace equation, which is modeled on the nonlocal
Trudinger equation. Our results are achieved by employing the expansion of positivity and measure theoretical techniques. In particular, the weak Harnack estimates highlight the nonlocal feature, as we only require the local positivity of weak supersolutions instead of the global one.
\end{abstract}

\section{Introduction}
\label{sec1}
\par

Let $E\subset\mathbb{R}^N$ be an open set and $T>0$. In this paper, we investigate the nonnegative weak supersolutions to the following doubly nonlinear nonlocal $p$-Laplace equations
\begin{align}
\label{1.1}
\partial_t(|u|^{p-2}u)+\mathcal{L}u=0\quad\text{in } E_T=E\times(0,T],
\end{align}
with $p>1$ and $\mathcal{L}$ being the nonlocal $p$-Laplace operator defined by
\begin{align*}
\mathcal{L} u(x,t)=\mathrm{P.V.} \int_{\mathbb{R}^N} K(x,y,t)|u(x,t)-u(y,t)|^{p-2}(u(x,t)-u(y,t))\,dy.
\end{align*}
The notation P.V. typically denotes the Cauchy principal value and the kernel function  $K(x, y, t):\mathbb{R}^N\times\mathbb{R}^N\times (0, T]\rightarrow [0,\infty)$ fulfills that
\begin{align*}
\frac{\Lambda^{-1}}{|x-y|^{N+sp}} \leq K(x,y,t)\equiv K(y,x,t) \leq \frac{\Lambda}{|x-y|^{N+sp}}\quad \mathrm{a.e. }\ x,y\in\mathbb{R}^N,
\end{align*}
for some $s\in(0,1)$ and $\Lambda \geq 1$. When $K(x,y,t)=\frac{1}{|x-y|^{N+sp}}$, Eq. \eqref{1.1} is nothing but the nonlocal
Trudinger equation
$$
\partial_t(|u|^{p-2}u)+(-\Delta)_p^s u=0,
$$
and it is the standard nonlocal heat equation if $p=2$.

\subsection{Function spaces and definitions}

For $p>1$ and $s\in(0,1)$, we introduce the fractional Sobolev space $W^{s,p}(\mathbb{R}^N)$ as
\begin{align*}
W^{s,p}\left(\mathbb{R}^N\right):=\left\{v \in L^p\left(\mathbb{R}^N\right): \int_{\mathbb{R}^N} \int_{\mathbb{R}^N} \frac{|v(x)-v(y)|^p}{|x-y|^{N+s p}}\,dxdy <\infty\right\},
\end{align*}
endowed with the norm
\begin{align*}
\|v\|_{W^{s,p}\left(\mathbb{R}^N\right)}:=\left(\int_{\mathbb{R}^N}|v|^p\,dx \right)^{\frac{1}{p}}+\left(\int_{\mathbb{R}^N} \int_{\mathbb{R}^N} \frac{|v(x)-v(y)|^p}{|x-y|^{N+sp}}\,dxdy\right)^{\frac{1}{p}}.
\end{align*}
In general, the above definitions also apply to the fractional Sobolev space $W^{s,p}(E)$ when the domain $E\subset \mathbb{R}^N$, and we define the corresponding seminorm as
\begin{align*}
[v]_{W^{s,p}(E)}=\left(\int_E \int_E \frac{|v(x)-v(y)|^p}{|x-y|^{N+sp}}\,dxdy \right)^{\frac{1}{p}}.
\end{align*}
The fractional Sobolev space with zero boundary values is defined as 
\begin{align*}
W_0^{s,p}(E):=\left\{v \in W^{s, p}\left(\mathbb{R}^N\right): v=0 \text { a.e. in } \mathbb{R}^N \backslash E\right\}.
\end{align*}

We also recall the tail space
\begin{align*}
L_{sp}^{p-1}(\mathbb{R}^{N}):=\left\{v \in L_{\rm{loc}}^{p-1}(\mathbb{R}^{N}): \int_{\mathbb{R}^N} \frac{|v(x)|^{p-1}}{1+|x|^{N+{sp}}}\,dx<+\infty\right\}, \quad p>1 \text{ and } \alpha>0,
\end{align*}
and the nonlocal parabolic tail is given by
\begin{align}
\label{1.2}
\mathrm{Tail}(v;x_0,R;t_0-S,t_0):=\left[\mint_{t_0-S}^{t_0}\left(R^{sp}\int_{\mathbb{R}^{N} \backslash K_R(x_0)} \frac{|v(x, t)|^{p-1}}{|x-x_0|^{N+sp}}\,dx\right)^{\frac{q}{p-1}}\,dt\right]^{\frac{1}{q}}
\end{align}
for some $q>p-1$. Moreover, we denote
\begin{align*}
\widetilde{\mathrm{Tail}}(v;x_0,R;t_0-S,t_0):=\left[\int_{t_0-S}^{t_0}\left(\int_{\mathbb{R}^{N} \backslash K_R(x_0)} \frac{|v(x, t)|^{p-1}}{|x-x_0|^{N+sp}}\,dx\right)^{\frac{q}{p-1}}\,dt\right]^{\frac{1}{q}}.
\end{align*}
It is apparent that $\mathrm{Tail}(v;x_0,R;t_0-S,t_0)$ is well-defined whenever $v\in L^q(t_0-S,t_0;L_{sp}^{p-1}(\mathbb{R}^N))$.

\medskip

We now give the definition of weak solutions that we work with.
\begin{definition}
	A function 
	\begin{align*}
	u\in C_{\mathrm{loc}}(0,T; L^p_{\mathrm{loc}}(E))\cap L^p_{\mathrm{loc}}(0,T;W^{s,p}_{\mathrm{loc}}(E))\cap L^{p-1}_{\mathrm{loc}}(0,T;L_{sp}^{p-1}(\mathbb{R}^N))
	\end{align*}
	is termed a local weak sub(super-) solution to \eqref{1.1} if for any compact set $K\subset E$ and any closed interval $[t_1,t_2]\subset (0,T]$, the following inequality holds
	\begin{align*}
	%\label{1.3}
	\int_K |u|^{p-2}u\varphi\,dx\Big|_{t_1}^{t_2}+\int_{t_1}^{t_2}\int_{K}-|u|^{p-2} u \partial_t \varphi\,dxdt+\int_{t_1}^{t_2} \mathcal{E}(u, \varphi, t)\,dt \leq (\geq) 0,
	\end{align*}
	for each nonnegative test function $\varphi \in W_{\mathrm{loc}}^{1,p}(0,T;L^p(K))\cap L^p_{\mathrm{loc}}(0,T;W_0^{s,p}(K))$, where
	\begin{align*}
	\mathcal{E}(u,\varphi,t):=\int_{\mathbb{R}^N}\int_{\mathbb{R}^N} K(x,y,t)|u(x,t)-u(y,t)|^{p-2}(u(x,t)-u(y,t))(\varphi(x,t)-\varphi(y,t))\,dydx.
	\end{align*}
\end{definition}
A function $u$ is said to be a weak solution provided that $u$ is both a local weak supersolution and a local weak subsolution. 
\subsection{Overview of related literature}

Eq. \eqref{1.1} can be regarded as a nonlocal version of the famous Trudinger equation
\begin{align}
\label{1.4}
\partial_t(|u|^{p-1}u)-\mathrm{div}(|\nabla u|^{p-2}\nabla u )=0,\quad p>1,
\end{align}
which is commonly classified as a more general class called doubly nonlinear parabolic equations, characterized by the nonlinearity present in both the time derivative and spatial gradient terms. The doubly nonlinear parabolic equations are of considerable interest because the mathematical structure encompasses mixed cases of degeneracy/singularity as well as their widespread applications, such as in shallow water flows \cite{ASD08}, the dynamics of glaciers \cite{M76}, and the model of flow in a gas-network dominated by friction \cite{LM18}. Further details on these models can be found in \cite{BDG23}. For the study of Harnack's estimates in doubly nonlinear equations, some significant results have been achieved. We refer the readers to \cite{B24, BS24,L18, V94,VV22, BHSS21,RSZ24, BDG23} and reference therein. The homogeneous form of doubly nonlinear equations is Eq. \eqref{1.4}, first introduced by Trudinger \cite{T68} in the study of a Harnack inequality for nonnegative weak solutions, utilizing Moser's iteration. This work was later extended to Trudinger's equation with more general operators in \cite{GV06} by proving mean value inequalities and incorporating Moser's logarithmic estimates. For what concerns the problem with a doubling Borel measure, Kinnunen-Kuusi \cite{KK07} developed a relatively simple approach to derive Harnack's estimates. With this result and the intrinsic scaling method, the H\"{o}lder continuity of nonnegative weak solutions to \eqref{1.4} was subsequently investigated in \cite{KSU12} for the degenerate case and in \cite{KSLU12} for the singular case. Alternatively, B\"{o}gelein-Duzaar-Liao \cite{BDL21} employed the expansion of positivity to study the interior and boundary regularity of weak solutions to \eqref{1.4} without any sign restriction. See also \cite{BDKS20, S10, LL22} for more related works.

Coming to the linear framework of \eqref{1.1} with $p=2$, namely
\begin{align}
\label{1.5}
\partial_t u-\mathrm{P.V.}\int_{\mathbb{R}^N} K(x,y,t)(u(x,t)-u(y,t))\,dy=0,
\end{align}
a weak Harnack inequality for globally positive solutions was proved in \cite{FK13}  and \cite{KS14}, under different conditions on the kernel function $K(x,y,t)$. Afterward, Str\"{o}mqvist \cite{S19} removed the requirement of global positivity for weak solutions, leading to the nonlocal version of the Harnack inequality; see \cite{K19} for the results concerning \eqref{1.5} with continuous data. Regarding the jumping kernels are not symmetric, a weak Harnack inequality and a Harnack inequality including the tail term were presented in \cite{KW22a} and \cite{KW22b}, respectively. For the linear mixed local and nonlocal parabolic equations, Garain-Kinnunen \cite{GK23} obtained a weak Harnack inequality for sign-changing supersolutions through a purely analytical argument. It is worth mentioning that the tail terms in the aforementioned results are bounded with respect to the time variable. Very recently, Kassmann-Weidner \cite{KW23} proposed a more natural definition of the tail in the parabolic setting as follows
%$L^1$-tail, defined as
\begin{align}
\label{1.6}
\int_{t_0-R^{2s}}^{t_0} \operatorname{Tail}(v(t);R,x_0)\,dt=\int_{t_0-R^{2s}}^{t_0} \int_{\mathbb{R}^N \backslash K_R\left(x_0\right)} \frac{|v(x,t)|}{\left|x-x_0\right|^{N+2s}}\,dxdt,
\end{align}
%which is more natural in the parabolic setting, 
and the improved local boundedness for subsolutions as well as a weak Harnack inequality for supersolutions with this $L^1$-tail were proved, which together imply the Harnack estimate. Moreover, the authors demonstrated the H\"{o}lder continuity for weak solutions under a hypothesis that the tail belongs to $L^{1+\varepsilon}$ in time and provided an example to show this assumption is optimal. Based on the tail definition \eqref{1.6}, Liao \cite{L24} derived several novel weak Harnack estimates of weak supersolutions to \eqref{1.5} relying on the measure theoretical approach that captured some nonlocal features previously overlooked, even in the context of elliptic problems. Besides, under certain integrability conditions on the time variable of the parabolic tail, the regularity results of weak solutions to the nonlocal $p$-Laplace parabolic equations were established in \cite{BK24} and \cite{Li24}.

Let us turn to the nonlocal Trudinger type equation \eqref{1.1}. As far as we know, the first contribution to the regularity theory was due to Banerjee-Garain-Kinnunen \cite{BGK22}. They obtained local boundedness estimates for the range $2<p<\infty$ and established a valuable algebraic inequality to treat the nonlocal term in deriving a reverse H\"{o}lder inequality. Building on this algebraic inequality and the expansion of positivity, Nakamura \cite{N23} discussed a Harnack inequality for the doubly nonlinear parabolic equation in the mixed local and nonlocal scenario.  Prasad \cite{P24} proved a weak Harnack estimate by imposing the global nonnegativity assumption on the weak solutions to \eqref{1.1}. For certain local properties, including the lower semicontinuity and pointwise estimates of supersolutions to general doubly nonlinear nonlocal parabolic equations, please refer to \cite{BGK23}.

Note that the form of the weak Harnack estimate obtained in \cite{P24} is analogous to the one of the local Trudinger equation, owing to the assumption that weak supersolutions are globally nonnegative. To emphasize the role of nonlocal operators, our work aims to establish a new type of weak Harnack estimates for locally positive weak supersolutions to \eqref{1.1}, under a regularity requirement on the tail with respect to the time variable. Consequently, our results will incorporate an optimal tail term, which can better exhibit the nonlocal features. Inspired by \cite{BK24, L24}, we employ the measure theoretical method and a suitable De Giorgi iteration to address the difficulties posed by the nonlocal term. It is remarked that the main idea is to derive an energy estimate for weak solutions to \eqref{1.1}, with the remainder of the proof in this work only relying on such energy estimates and selecting appropriate test functions, without using the equation anymore. In addition, we also establish a weak Harnack inequality in the time-space cylinders with a higher integral exponent than the known results.

\subsection{Main results}
Throughout the paper, let $\mathscr{Q}:=K_R(x_0)\times(T_1,T_2]$ denote a reference cylinder included in $E_T$. 

\smallskip
We state our first main result about a weak Harnack inequality as follows.
\begin{theorem}
\label{thm-1-2}
Let $p>1$ and $u$ be a local weak supersolution to \eqref{1.1} that satisfies $u\geq 0$ in $\mathscr{Q}$. Assume further that $u\in L^q_{\rm{loc}}(0,T;L^{sp}_{p-1}(\mathbb{R}^N))$ with $q>p-1$. Suppose $K_{4\rho}(x_0)\times(t_0,t_0+4(8\rho)^{sp}]\subset \mathscr{Q}$. Then there exist constants $\vartheta,\eta$ in $(0,1)$ depending only on $N,p,s,q,\Lambda$, such that
\begin{align*}
&\quad\left(\frac{\rho}{R}\right)^{\frac{sp}{p-1}}\mathrm{Tail}(u_-;x_0,R;t_0,t_0+4(8\rho)^{sp})+\operatorname*{ess\inf}_{K_{2\rho}(x_0)} u(\cdot,t)\\
&\geq \eta\left(\mint_{K_\rho(x_0)} u^\beta(\cdot,t_0)\,dx\right)^{\frac{1}{\beta}}
\end{align*}
for all times
\begin{align*}
t\in \left(t_0+\frac{1}{2}(8\rho)^{sp},t_0+ 2(8\rho)^{sp}\right].
\end{align*}
\end{theorem}

The upcoming statement is a weak Harnack estimate that improves upon the integral power of the first one, with its form presented in the time-space cylinders.

\begin{theorem}
\label{thm-1-3}
Let $p>1$ and $u$ be a local weak supersolution to \eqref{1.1} that satisfies $u\geq 0$ in $\mathscr{Q}$. Assume further that $u\in L^q_{\rm{loc}}(0,T;L^{sp}_{p-1}(\mathbb{R}^N))$ with $q>p-1$. Suppose $K_{2\rho}(x_0)\times(t_0-(2\rho)^{sp},t_0+4(4\rho)^{sp}]\subset \mathscr{Q}$.   Then there exists a constant $\eta\in(0,1)$ depending only on $N,p,s,q,\Lambda$, such that
\begin{align*}
&\quad \left(\frac{\rho}{R}\right)^{\frac{sp}{p-1}}\mathrm{Tail}(u_-;x_0,R;t_0-(2\rho)^{sp},t_0+4(4\rho)^{sp})+\operatorname*{ess\inf}_{K_{\rho}(x_0)} u(\cdot,t)\\
&\geq \eta\left(\mint\!\!\mint_{K_\rho(x_0)\times(t_0-\rho^{sp},t_0)} u^{p-1}\,dxdt\right)^{\frac{1}{p-1}}
\end{align*}
for all times
\begin{align*}
t\in \left(t_0+\frac{3}{4}(4\rho)^{sp},t_0+ (4\rho)^{sp}\right].
\end{align*}
\end{theorem}
\begin{remark}
\label{rem-1-4}
If we assume $u\geq m>0$ in $\mathbb{R}^N\times(0, T)$, then by following a proof similar to that of Proposition 3.3 in \cite{N23}, the exponent $p-1$ in Theorem \ref{thm-1-3} can be improved up to $\sigma$, where $\sigma<\frac{N+sp}{N}(p-1)$.
\end{remark}
\begin{remark}
It is worth noting that the weak Harnack inequality in Theorem \ref{thm-1-3} does not imply the one in Theorem \ref{thm-1-2}, since Theorem \ref{thm-1-3} requires the solution property for $t<t_0$, whereas Theorem \ref{thm-1-2} does not.
\end{remark}

%The nonlocal tail term can be eliminated in Theorems \ref{thm-1-2} and \ref{thm-1-3} if the weak supersolutions of \eqref{1.1} are globally positive.

The plan of this paper is as follows. In Section \ref{sec2}, we collect some basic notations and auxiliary tools. The energy estimate and several preliminary lemmas will be given in Section \ref{sec3}. Section \ref{sec4} is devoted to showing the expansion of positivity which plays an important role in deriving our main results. Finally, we establish the weak Harnack estimates Theorem \ref{thm-1-2} and Theorem \ref{thm-1-3} in Section \ref{sec5}.

\section{Preliminaries}
\label{sec2}

\subsection{Notation}
We first display some standard notations that will be needed. Let $K_\rho(x_0)$ denote the ball with radius $\rho$ centered at $x_0\in\mathbb{R}^N$. For $(x_0,t_0)\in \mathbb{R}^N\times\mathbb{R}$, the backward cylinders are given by
%\begin{align*}
%(x_0, t_0)+Q_{R,S}:=K_R(x_0) \times\left(t_0-S,t_0\right]
%\end{align*}
%and
\begin{align*}
(x_0, t_0)+Q_\rho(\theta):= K_\rho(x_0)\times(t_0-\theta\rho^{sp},t_0).
\end{align*}
We shall omit $(x_0,t_0)$ from the above symbol when there is no ambiguity. The scaling parameter $\theta$ will also be omitted if $\theta=1$.

For $k\in\mathbb{R}$, let
\begin{align*}
(u-k)_+=\max\{u-k,0\} \quad \textmd{and} \quad  (u-k)_-=\max\{-(u-k),0\}.
\end{align*}
We use the notation
\begin{align*}
d\mu=d\mu(x,y,t)=K(x,y,t)\,dxdy
\end{align*}
and 
\begin{align*}
A_\pm(u;x_0,\rho;k)=\{x\in B_\rho(x_0):(u(x,t)-k)_\pm>0\}.
\end{align*}
%and
%\begin{align*}
%U(x,y,t):=|u(x,t)-u(y,t)|^{p-2}(u(x,t)-u(y,t)).
%\end{align*}
The generic constants will be denoted by $\gamma$ and may differ from line to line.

\subsection{Technical lemmas}

For $a,b\in\mathbb{R}$, we define functions $\bm{h}_\pm$ as
\begin{align}
\label{2.1}
\bm{h}_\pm(a,b):=\pm (p-1) \int_b^a|s|^{p-2}(s-a)_{\pm}\,ds.
\end{align}
It can be easily checked that $\bm{h}_\pm\geq 0$. 

The following lemma is an inequality retrieved from \cite[Lemma 2.2]{AF89} for the case $0<\alpha<1$ and from \cite[inequality (2.4)]{GM86} for the case $\alpha>1$. 
\begin{lemma}
\label{lem-2-1}
For all $\alpha>0$, there exists a constant $\gamma(\alpha)>0$, such that whenever $a,b\in\mathbb{R}$, we have 
\begin{align*}
\frac{1}{\gamma}\left|| b|^{\alpha-1} b-|a|^{\alpha-1} a\right|\leq(|a|+|b|)^{\alpha-1}
|b-a|\leq \gamma\left||b|^{\alpha-1} b-|a|^{\alpha-1} a \right|.
\end{align*}

\end{lemma}
By Lemma \ref{lem-2-1}, the following inequality holds.

\begin{lemma} [Lemma 2.2, \cite{BK24}]
\label{lem-2-2}
Let $p\geq 1$. There exists a constant $\gamma(p)>0$, such that whenever $a,b\in\mathbb{R}$, we have
\begin{align*}
\frac{1}{\gamma}(|a|+|b|)^{p-2}(a-b)_{ \pm}^2 \leq \bm{h}_{ \pm}(a, b) \leq \gamma(|a|+|b|)^{p-2}(a-b)_{ \pm}^2.
\end{align*}
\end{lemma}

We now give an embedding lemma, extracted from \cite[Lemma 2.2]{BK24}.

\begin{lemma} 
	\label{lem-2-3}
	Let $s\in(0,1]$, $p\geq 1$ and $1<m\leq\max\{p,2\}$. For each function
	$$u\in L^{\infty}\left(I; L^m(K_R)\right) \cap L^p\left(I; W^{s, p}(K_R)\right),$$
	there exists a constant $\gamma(N,s,p,\bar{q},\bar{r})$ such that
	\begin{align*}
	\|u\|_{L^{\bar{q}, \bar{r}}\left(K_R \times I\right)}^{\frac{N \bar{q} \bar{r}}{sp \bar{q}+N\bar{r}}} \leq \gamma\left(\int_I[u(\cdot,t)]_{W^{s, p}\left(K_R\right)}^p \,dt+R^{-s p}\|u\|_{L^p\left(I;L^p\left(K_R\right)\right)}^p+\underset{t \in I}{\operatorname{ess} \sup }\|u(\cdot, t)\|_{L^m\left(K_R\right)}^m\right), 
	\end{align*}
	where the positive numbers $\bar{q}$ and $\bar{r}$ fulfilling
	$$
	\frac{1}{\bar{q}}+\frac{1}{\bar{r}}\left(\frac{s p}{N}+\frac{p}{m}-1\right)=\frac{1}{m},
	$$
	and their admissible ranges are
	\begin{align*}
	\left\{\begin{array}{lll}
		\bar{q} \in\left[m, \frac{Np}{N-sp}\right], & \bar{r}\in \left[p, \infty\right) & \text { if } N>sp, \\
		\bar{q} \in[m, \infty), & \bar{r} \in\left(\frac{msp}{N}+p-m, \infty\right) & \text { if } N=sp, \\
		\bar{q} \in[m, \infty], & \bar{r} \in\left[\frac{msp}{N}+p-m, \infty\right) & \text { if } N<sp.
	\end{array}\right.
	\end{align*}
\end{lemma}

We end this section by listing an iteration lemma, extracted from \cite[Chapter \uppercase\expandafter{\romannumeral1}, Lemma 4.2]{D93}.

\begin{lemma} 
	\label{lem-2-4}
	Suppose $\{Y_j\}_{j=0}^\infty$ and $\{Z_j\}_{j=0}^\infty$  are sequences of positive numbers that satisfy
	\begin{align*}
	Y_{j+1} \leq K b^j\left(Y_j^{1+\delta}+Z_j^{1+\kappa} Y_j^\delta\right) \quad \text { and } \quad Z_{j+1} \leq K b^j\left(Y_j+Z_j^{1+\kappa}\right),\quad j=0,1,2, \ldots
	\end{align*}
	for some constants $K,b>1$, and $\delta,\kappa>0$. If 
	\begin{align*}
	Y_0+Z_0^{1+\kappa} \leq(2 K)^{-\frac{1+\kappa}{\sigma}} b^{-\frac{1+\kappa}{\sigma^2}},
	\end{align*}
	with $\sigma=\min\{\kappa,\delta\}$, then $Y_j\rightarrow 0$ and $Z_j\rightarrow 0$ as $j\rightarrow \infty$.
\end{lemma}

\section{Crucial lemmas}
\label{sec3}
In this section, several quantitative estimates will be given. We begin with a Caccioppoli-type inequality, for which the proof follows a similar approach to that in \cite{BDL21, Li24, BK24}. The proof is omitted here.

\begin{lemma}
\label{lem-3-1}
Let $p>1$ and $u$ be a local weak subsolution to \eqref{1.1}. Suppose the cylinder $(t_0-S,t_0)\times K_R(x_0)$ is included in $E_T$. Let $\psi(\cdot,t)\in(0,1]$ be a piecewise smooth cutoff function compactly supported in $K_R(x_0)$ for any $t \in\left(t_0-S, t_0\right)$. There exists a constant $\gamma(p,\Lambda)>0$, such that for any $k\in\mathbb{R}$, we have
\begin{align}
\label{3.1}
&\operatorname*{ess \sup}_{t_0-S<t<t_0} \int_{K_R(x_0)\times\{t\}}  \bm{h}_+(u,k)\psi^p(x,t)\,dx\nonumber\\
&+\int_{t_0-S}^{t_0} \int_{K_R(x_0)} \int_{K_R(x_0)} \min \left\{\psi^p(x,t), \psi^p(y, t)\right\} \frac{\left|w_+(x,t)-w_+(y,t)\right|^p}{|x-y|^{N+sp}}\,dxdydt \nonumber\\
& +\int_{t_0-S}^{t_0} \int_{K_R(x_0)} \psi^p w_+(x,t)\int_{K_R(x_0)} \frac{w_-^{p-1}(y, t)}{|x-y|^{N+sp}}\,dydxdt\nonumber\\
\leq& \gamma \int_{t_0-S}^{t_0} \int_{K_R(x_0)}\bm{h}_+(u,k)|\partial_t\psi^p|\,dxdt+\int_{K_R(x_0)\times\{t_0-S\}}  \bm{h}_+(u,k)\psi^p\,dx\nonumber\\
&+\gamma \int_{t_0-S}^{t_0} \int_{K_R(x_0)}\int_{K_R(x_0)} \max \left\{w_+^p(x,t), w_+^p(y,t)\right\} \frac{|\psi(x,t)-\psi(y,t)|^p}{|x-y|^{N+sp}}\,dxdydt \nonumber\\
&+\gamma\int_{t_0-S}^{t_0}\int_{\mathbb{R}^N \backslash K_R(x_0)}\int_{K_R(x_0)} \frac{w_+^{p-1}(y,t)}{|x-y|^{N+sp}}w_+(x,t)\psi^p(x,t)\,dxdydt,
\end{align}
where $w(x,t):=u(x,t)-k$ and the function $\bm{h}_+$ is defined as in \eqref{2.1}.
\end{lemma}

The next one is a De Giorgi-type lemma, saying that if some measure density condition arrives, then a pointwise estimate can be obtained in a smaller cylinder.

\begin{lemma}
\label{lem-3-2}
Let $p>1$ and $u$ be a local weak supersolution satisfying $u\geq 0$ in $\mathscr{Q}$.  Assume further that $u\in L^q_{\rm{loc}}(0,T;L^{sp}_{p-1}(\mathbb{R}^N))$ with $q>p-1$. Suppose $M>0$, $\delta\in(0,1)$, and $(x_0,t_0)+Q_{2\rho}(\delta)\subset \mathscr{Q}$. If there holds
\begin{align*}
|\{u\leq M\} \cap(x_0,t_0)+ Q_{\rho}(\delta)| \leq \nu|Q_{\rho}(\delta)|
\end{align*}
for some constant $\nu\in(0,1)$ depending on $N,p,s,q,\Lambda$ and $\delta$, then we have either
\begin{align*}
\left(\frac{\rho}{R}\right)^{\frac{sp}{p-1}}\mathrm{Tail}(u_-;x_0,R;t_0-\delta\rho^{sp},t_0)>M,
\end{align*}
or
\begin{align*}
u\geq\frac{1}{2}M \quad \text { a.e. in } (x_0,t_0)+Q_{\frac{1}{2}\rho}(\delta).
\end{align*}
Moreover, the constant $\nu$ can be more precisely written as $\nu=\bar{\nu}\delta^l$ with some $\bar{\nu}\in(0,1)$ and $l>0$ depending on $N,p,s,q,\Lambda$. 

\begin{proof}
Assume that $(x_0,t_0)=(0,0)$. For $j=0,1,2,\ldots$, taking
\begin{align*}
k_j=\frac{M}{2}+\frac{M}{2^{j+1}},\quad \widehat{k_j}=\frac{k_j+k_{j+1}}{2},\quad \overline{k}_j=\frac{k_{j+1}+\widehat{k}_j}{2}
\end{align*}
and
\begin{align*}
\rho_j=\frac{\rho}{2}+\frac{\rho}{2^j}, \quad \widehat{\rho}_j=\frac{\rho_j+\rho_{j+1}}{2}. 
% \widetilde{\rho}_j=\frac{3\rho_j+\rho_{j+1}}{4}.
\end{align*}
Consider the domains
\begin{align*}
K_j=K_{\rho_j}, \quad \widehat{K}_j=K_{\widehat{\rho}_j}, \ \ j=0,1,2\ldots,
\end{align*}
and
\begin{align*}
Q_j=K_j \times\left(-\delta \rho_j^{sp}, 0\right], \quad \widehat{Q}_j=\widehat{K}_j \times\left(-\delta \widehat{\rho}_j^{sp}, 0\right], \ \ j=0,1,2\ldots.
\end{align*}
Consider a cutoff function $\psi(x,t)\in(0,1]$ defined in $Q_j$ such that
\begin{align*}
\psi=1 \text{ in } Q_{j+1},\quad \psi=0 \text{ on } \partial_p \widehat{Q}_j,\quad |\nabla \psi|\leq \frac{2^j}{\rho},\quad |\partial_t\psi|\leq\frac{2^{spj}}{\delta\rho^{sp}}.
\end{align*}
Additionally, Lemma \ref{lem-2-2} reads that
\begin{align*}
\frac{1}{\gamma}(|u|+\widehat{k}_j)^{p-2}(u-\widehat{k}_j)_-^2\leq \bm{h}_-(u,\widehat{k}_j)\leq \gamma(|u|+\widehat{k}_j)^{p-2}(u-\widehat{k}_j)_-^2.
\end{align*}
Based on this inequality, using the energy estimate \eqref{3.1} in the setting above with $w_-=(u-\widehat{k_j})_-$ gives 
\begin{align}
\label{3.2}
I_{0,1}+I_{0,2}&:=\operatorname*{ess\sup}_{-\delta \rho_{j+1}^{sp}<t<0} \int_{K_{j+1}} (|u|+\widehat{k}_j)^{p-2}w_{-}^2(x,t)\,dx\nonumber\\
&\quad+\iint_{Q_{j+1}}\int_{K_{j+1}}\frac{|w_-(x,t)-w_-(y,t)|^p}{|x-y|^{N+sp}} \,dxdydt \nonumber\\
&\leq\gamma \iint_{Q_j}(|u|+\widehat{k}_j)^{p-2}w_-^2(x,t)|\partial_t \psi^p|\,dxdt\nonumber\\
&\quad+\gamma\iint_{Q_j} \int_{K_j} \max \left\{w_-^p(x, t),w_-^p(y,t)\right\}|\psi(x,t)-\psi(y,t)|^p\,d\mu dt \nonumber\\
&\quad+\gamma \int_{-\delta\rho_j^{sp}}^{0}\int_{\mathbb{R}^N \backslash K_j}\int_{K_j} \frac{w_-^{p-1}(y,t)}{|x-y|^{N+sp}}w_-(x,t)\psi^p(x,t)\,dxdydt\nonumber\\
&=:I_1+I_2+I_3.
\end{align}

Denote 
\begin{align*}
y_j=\frac{|\{u(x,t)\leq k_j\}\cap Q_j|}{\left|Q_j\right|}
\end{align*}
and 
\begin{align*}
z_j=\left[\mint_{-\delta \rho_j^{s p}}^0\left(\frac{\left|A_{-}\left(u;0,\rho_j; k_j\right)\right|}{\left|K_j\right|}\right)^{\frac{q}{q-(p-1)}}\, d t\right]^{\frac{q-(p-1)}{q(1+\kappa)}},
\end{align*}
where $\kappa=\frac{sp}{N}\left[\frac{q-(p-1)}{q}\right]$.
Next, we evaluate $I_1$--$I_3$ in \eqref{3.2} separately. The estimates of $I_1$ and $I_2$ are standard, and we employ the properties of $\psi(x,t)$ to obtain
\begin{align*}
I_1&\leq \iint_{Q_j}(|u|+\widehat{k}_j)^{p-1}w_-(x,t)|\partial_t \psi^p|\,dxdt\\
&\leq \gamma 2^{spj}\frac{M^p}{\delta\rho^{sp}}y_j|Q_j|
\end{align*}
%Again using the properties of the function $\psi(x,t)$, we treat $I_2$ by
and
\begin{align*}
I_2&\leq \gamma 2^{pj}\frac{M^p}{\rho^p}\int_{-\delta\rho_j^{sp}}^{0}\int_{K_j}\int_{K_j}\frac{\chi_{\{u(x,t)<k_j\}}}{|x-y|^{N+sp-p}}\,dxdydt\\
&\leq\gamma 2^{pj} \frac{M^p}{\rho^{sp}}y_j|Q_j|.
\end{align*}
For $I_3$, notice that
\begin{align*}
\frac{|y|}{|y-x|}\leq 1+\frac{|x|}{|y-x|}\leq 1+\frac{\widehat{\rho}_j}{\rho_j-\widehat{\rho}_j}\leq 2^{j+4}
\end{align*}
holds for all $|x|\leq\widehat{\rho}_j$ and  $|y|\geq\rho_j$, then we have by the assumption $u\geq 0$ in $\mathscr{Q}$ that
\begin{align*}
I_3&\leq \gamma 2^{j(N+sp)} M  \int_{-\delta\rho_j^{sp}}^{0}\int_{\mathbb{R}^N \backslash K_j}\int_{K_j}\frac{w_-^{p-1}(y,t)}{|y|^{N+sp}}\chi_{\{u(x,t)<\widehat{k}_j\}}\,dxdydt\\
& \leq\gamma 2^{j(N+sp)}\left( \frac{M^p}{\rho^{sp}}y_j|Q_j|+M\int_{-\delta\rho_j^{sp}}^{0}\int_{\mathbb{R}^N \backslash K_R}\int_{K_j}\frac{u_-^{p-1}(y,t)}{|y|^{N+sp}}\chi_{\{u(x,t)<\widehat{k}_j\}}\,dxdydt\right)\\
&=:\gamma 2^{j(N+sp)} \frac{M^p}{\rho^{sp}}y_j|Q_j|+I_{3,2}.
\end{align*}
By virtue of H\"{o}lder's inequality, we estimate $I_{3,2}$ as
\begin{align*}
	I_{3,2} &\leq  \gamma 2^{j(N+sp)}M [\widetilde{\mathrm{Tail}}\left(u_-;0,R;-\delta\rho^{sp},0\right)]^{p-1}\\
	&\quad \times\left(\int_{-\delta \rho_j^{s p}}^0\left|A_-\left(u;0, \rho_j; \widehat{k}_j\right)\right|^{\frac{q}{q-(p-1)}} \mathrm{d} t\right)^{\frac{q-(p-1)}{q}} \\
	 &\leq \gamma 2^{j(N+sp)}M^p \delta^{\frac{p-1}{q}} \rho^{-N \kappa}\left(\int_{-\delta \rho_j^{s p}}^0\left|A_-\left(u;0, \rho_j; k_j\right)\right|^{\frac{q}{q-(p-1)}} \mathrm{d} t\right)^{\frac{q-(p-1)}{q}} \\
	& \leq \gamma 2^{j(N+sp)}M^p \delta^{\frac{p-1}{q}} \rho^{-N \kappa}\left(\delta^{\frac{q-(p-1)}{q}} \rho_j^{N+N \kappa} z_j^{1+\kappa}\right) \\
	 &\leq \gamma\frac{2^{j(N+s p)}}{\rho^{s p}}M^p z_j^{1+\kappa}\left|Q_j\right|,
\end{align*}
if enforcing
\begin{align*}
\frac{\rho^{\frac{N\kappa}{p-1}}}{\delta^{\frac{1}{q}}}\widetilde{\mathrm{Tail}}(u_-;0,R;-\delta\rho^{sp},0)=\left(\frac{\rho}{R}\right)^{\frac{sp}{p-1}}\mathrm{Tail}(u_-;0,R;-\delta\rho^{sp},0)\leq M.
\end{align*}

Whereas the first part on the left-hand side of \eqref{3.2} can be handled by the definitions of $\widehat{k}_j$, $\overline{k}_j$ and $u\geq 0$ in $\mathscr{Q}$ as
\begin{align*}
\operatorname*{ess\sup}_{-\delta \rho_{j+1}^{sp}<t<0} \int_{K_{j+1}} (|u|+\widehat{k}_j)^{p-2}w_-^2(x,t)\,dx\geq\frac{\gamma}{2^{j|p-2|}}\operatorname*{ess\sup}_{-\delta\rho_{j+1}^{sp}<t<0} \int_{K_{j+1}}(u-\overline{k}_j)_-^p(x,t)\,dx.
\end{align*}
Merging the above estimates with Lemma 2.3 in \cite{BK24}, we arrive at
\begin{align*}
&\quad\operatorname*{ess\sup}_{-\delta\rho_{j+1}^{sp}<t<0} \int_{K_{j+1}}(u-\overline{k}_j)_-^p(x,t)\,dx\\
&\quad+\iint_{Q_{j+1}}\int_{K_{j+1}}\frac{|(u-\overline{k}_j)_-(x,t)-(u-\overline{k}_j)_-(y,t)|^p}{|x-y|^{N+sp}} \,dxdydt \nonumber\\
&\leq\gamma 2^{j(N+2p+2)}\left(\frac{M^p}{\delta\rho^{sp}}y_j|Q_j|+\frac{M^p}{\rho^{sp}} z_j^{1+\kappa}\left|Q_j\right|\right).
\end{align*}
The following steps are analogous to Lemma 4.1 in \cite{BK24}. For the sake of clarity in the proof, we provide the details here. By utilizing H\"{o}lder's inequality and the embedding Lemma \ref{lem-2-3} with $m=p$, $\bar{q}=p(1+\frac{sp}{N})$ and $\bar{r}=p(1+\frac{sp}{N})$, we have
\begin{align}
\label{3.3}
 y_{j+1}\left|Q_{j+1}\right| &\leq \iint_{Q_{j+1}} \frac{\left(u-\overline{k}_j\right)_{-}^p}{\left(\overline{k}_j-k_{j+1}\right)^p}\,dxdt\nonumber\\
& \leq \gamma 2^{j p}M^{-p}\iint_{Q_{j+1}} \left(u-\overline{k}_j\right)_{-}^p\,dxdt\nonumber\\
& \leq \gamma 2^{j p}M^{-p}\left|Q_j\right|^{\frac{s p}{N+s p}} y_j^{\frac{s p}{N+s p}}\left\|\left(u-\overline{k}_j\right)_{-}\right\|_{L^{p\left(1+\frac{s p}{N}\right)}\left(Q_{j+1}\right)}^p \nonumber\\
& \leq \gamma 2^{j p}M^{-p}\left|Q_j\right|^{\frac{s p}{N+s p}} y_j^{\frac{s p}{N+s p}}\left[\int_{-\delta \rho_{j+1}^{s p}}^0\left[\left(u-\overline{k}_j\right)_-(\cdot, t)\right]_{W^{s,p}\left(K_{j+1}\right)}^p\,dt\right.\nonumber\\
&\quad \left.+\rho_{j+1}^{-s p}\left\|\left(u-\overline{k}_j\right)_{-}\right\|_{L^p(Q_{j+1})}^p+\underset{t \in\left(-\delta \rho_{j+1}^{s p}, 0\right)}{\operatorname{ess} \sup } \int_{K_{j+1}}\left(u-\overline{k}_j\right)_{-}^p(x,t)\,dx \right] \nonumber\\
& \leq \gamma 2^{j p}M^{-p}\left|Q_j\right|^{\frac{s p}{N+s p}} y_j^{\frac{s p}{N+s p}} 2^{j(N+2p+2)}\left(\frac{M^p}{\delta\rho^{sp}}y_j|Q_j|+ \frac{M^p}{\rho^{sp}} z_j^{1+\kappa}\left|Q_j\right|\right).
%=: c 2^{h p}(\xi \omega)^{-p}\left|Q_{\rho h}^{(\delta)}\right|^{\frac{s p}{n+s p}} y_h^{\frac{s p}{n+s p}} J .
\end{align}
Dividing both sides of \eqref{3.3} by $|Q_{j+1}|$ yields
\begin{align}
\label{3.4}
y_{j+1} \leq \frac{\gamma}{\delta} 2^{j(N+3p+2)} \delta^{\frac{s p}{N+s p}}\left(y_j^{1+\frac{s p}{N+sp}}+z_j^{1+\kappa} y_j^{\frac{s p}{N+s p}}\right).
\end{align}

Let us now estimate $z_{j+1}$. An application of Lemma \ref{lem-2-3} with $m=p$, $\bar{q}=p(1+\kappa)$ and $\bar{r}=\frac{qp(1+\kappa)}{q-(p-1)}$ gives 
\begin{align*}
	\delta^{\frac{q-(p-1)}{q(1+\kappa)}} \rho_j^N z_{j+1} & \leq\left[\int_{-\delta \rho_{j+1}^{s p}}^0\left(\int_{K_{j+1}}\left(\frac{\left(u-\overline{k}_j\right)_{-}}{\left(\overline{k}_j-k_{j+1}\right)}\right)^{p(1+\kappa)}\,dx\right)^{\frac{q}{q-(p-1)}} \,dt\right]^{\frac{q-(p-1)}{q(1+\kappa)}} \\
	& \leq \gamma \frac{2^{j p}}{M^p}\left\|\left(u-\overline{k}_j\right)_{-}\right\|_{L^{p(1+\kappa),\frac{qp(1+\kappa)}{q-(p-1)}}(Q_{j+1})}^{p} \\
	& \leq \gamma \frac{2^{jp}}{M^p} \left[\int_{-\delta \rho_{j+1}^{s p}}^0\left[\left(u-\overline{k}_j\right)_-(\cdot, t)\right]_{W^{s,p}\left(K_{j+1}\right)}^p\,dt\right.\nonumber\\
	&\quad \left.+\rho_{j+1}^{-s p}\left\|\left(u-\overline{k}_j\right)_{-}\right\|_{L^p(Q_{j+1})}^p+\underset{t \in\left(-\delta \rho_{j+1}^{s p}, 0\right)}{\operatorname{ess} \sup } \int_{K_{j+1}}\left(u-\overline{k}_j\right)_{-}^p(x,t)\,dx \right] \\
	& \leq \gamma \frac{2^{j(N+3p+2)}}{\delta \rho^{s p}} y_j\left|Q_j\right|+\gamma \frac{2^{j(N+3p+2)}}{\rho^{sp}} z_j^{1+\kappa}\left|Q_j\right|,
\end{align*}
which directly leads to
\begin{align}
\label{3.5}
z_{j+1} \leq \frac{\gamma}{\delta} 2^{j(N+3 p+2)} \delta^{1-\frac{q-(p-1)}{q(1+\kappa)}}\left(y_j+z_j^{1+\kappa}\right).
\end{align}
From \eqref{3.4} and \eqref{3.5}, we obtain the relations

\begin{align*}
& y_{j+1} \leq \frac{\gamma}{\delta} 2^{j(N+3p+2)} \delta^\ell\left(y_j^{1+\frac{s p}{N+sp}}+z_j^{1+\kappa} y_j^{\frac{s p}{N+s p}}\right), \\
& z_{j+1} \leq \frac{\gamma}{\delta} 2^{j(N+3p+2)} \delta^\ell\left(y_j+z_j^{1+\kappa}\right),
\end{align*}\\
where $\gamma$ is a constant just depending on $N,p,s,q,\Lambda$ and
\begin{align}
\label{3.6}
\ell=\min \left\{\frac{s p}{N+s p}, 1-\frac{q-(p-1)}{q(1+\kappa)}\right\}.
\end{align}

Moreover, we can see
\begin{align*}
z_0 \leq\left(\mint_{-\delta \rho_0^{sp}}^0 \frac{\left|A_{-}\left(u;0,\rho_0; k_0\right)\right|}{\left|K_0\right|}\,\right)^{\frac{q-(p-1)}{q(1+\kappa)}} \leq y_0^{\frac{q-(p-1)}{q(1+\kappa)}}.
\end{align*}
Thus there holds
\begin{align*}
y_0+z_0^{1+\kappa} \leq 2 \nu^{\frac{q-(p-1)}{q}},
\end{align*}
provided that $y_0 \leq \nu<1$. Making use of Lemma \ref{lem-2-4}, it indicates that $y_j\rightarrow 0$ as $j\rightarrow\infty$ if
\begin{align*}
y_0+z_0^{1+\kappa}\leq\left(\frac{2\gamma}{\delta^{1-\ell}}\right)^{-\frac{1+\kappa}{\varsigma}}2^{-(N+3p+2)\frac{1+\kappa}{\varsigma^2}},
\end{align*}
where $\varsigma=\min\{\kappa,\frac{sp}{N+sp}\}$. Finally, we get the desired result by taking
\begin{align*}
\nu \leq\left(\frac{\delta^{1-\ell}}{4 \gamma}\right)^{\frac{q(1+\kappa)}{(q-(p-1)) \varsigma}} 2^{-(N+3p+2) \frac{q(1+\kappa)}{(q-(p-1)) \varsigma^2}}.
\end{align*}
\end{proof}
\end{lemma}

The following lemma is from the measure density condition at a given time to deduce a pointwise estimate at later times.

\begin{lemma}
	\label{lem-3-3}
Let $p>1$ and $u$ be a local weak supersolution to \eqref{1.1} that satisfies $u\geq 0$ in $\mathscr{Q}$. Assume further that $u\in L^q_{\rm{loc}}(0,T;L^{sp}_{p-1}(\mathbb{R}^N))$ with $q>p-1$. Let $M>0$ be a constant. If one can find $\delta,\nu_0\in(0,1)$ depending only on $N,p,s,q,\Lambda$, such that  
\begin{align}
\label{3.7}
|\{u(\cdot,t_0)\leq M\}\cap K_{2\rho}(x_0)|<\nu_0|K_{2\rho}|,
\end{align}
then there holds either
\begin{align*}
\left(\frac{\rho}{R}\right)^{\frac{sp}{p-1}}\mathrm{Tail}(u_-;x_0,R;t_0,t_0+\delta\rho^{sp})>M,
\end{align*}
or
\begin{align*}
u \geq \frac{1}{4} M \quad \text { a.e. in } K_{\frac{1}{2} \rho}(x_0) \times\left(t_0+\frac{3}{4} \delta \rho^{sp}, t_0+\delta \rho^{sp}\right],
\end{align*}
with imposing $K_{2 \rho}(x_0) \times(t_0, t_0+\delta \rho^{sp}]\subset\mathscr{Q}$.
\end{lemma}
\begin{proof}
Without loss of generality, we may assume $(x_0,t_0)=(0,0)$. For $j=0,1,2,\ldots$, we denote $\rho_j=\rho+\frac{\rho}{2^j}$ and use the identical notations of $k_j,\widehat{k}_j, \widehat{\rho}_j,K_j,\widehat{K}_j$, as in Lemma \ref{lem-3-2}. Besides, the cylinders with fixed height $\delta\rho^{sp}$ are denoted by 
\begin{align*}
Q_j=K_j \times\left(0,\delta \rho^{sp}\right], \quad \widehat{Q}_j=\widehat{K}_j \times\left(0,\delta \rho^{sp}\right],
%\widetilde{Q}_j=\widetilde{K}_j \times\left(0,\delta \rho^{sp}\right], 
\end{align*}
where $\delta\in(0,1)$ will be chosen later. We adopt the cutoff function $\psi(x,t)\equiv\psi(x)$ in $K_j$, which meets 
\begin{align*}
0<\psi\leq 1,\quad\psi(x)=0 \text{ in } \mathbb{R}^N\backslash \widehat{K}_j, \quad \psi(x)= 1 \text{ in } {K}_{j+1}, \quad |\nabla\psi|\leq \frac{2^{j+4}}{\rho}.
\end{align*}
Exploiting the energy estimate \eqref{3.1} with this structure leads to
\begin{align}
\label{3.8}
&\quad \operatorname*{ess\sup}_{0<t<\delta \rho^{sp}} \int_{K_{j+1}} (|u|+\widehat{k}_j)^{p-2}(u-\widehat{k}_j)_{-}^2\,dx\nonumber\\
&\quad+\iint_{Q_{j+1}}\int_{K_{j+1}}\frac{|(u-\widehat{k}_j)_-(x,t)-(u-\widehat{k}_j)_-(y,t)|^p}{|x-y|^{N+sp}} \,dxdydt \nonumber\\
&\leq \gamma\iint_{Q_j} \int_{K_j} \max \{(u-\widehat{k}_j)_-^p(x, t),(u-\widehat{k}_j)_-^p(y,t)\}|\psi(x)-\psi(y)|^p\,d\mu dt \nonumber\\
&\quad+\gamma \int_{0}^{\delta\rho^{sp}}\int_{\mathbb{R}^N \backslash K_j}\int_{K_j} \frac{(u-\widehat{k}_j)_-^{p-1}(y,t)}{|x-y|^{N+sp}}(u-\widehat{k}_j)_-(x,t)\psi^p\,dxdydt\nonumber\\
&\quad+\gamma \int_{K_j}(|u(x,0)|+\widehat{k}_j)^{p-2}(u(x,0)-\widehat{k}_j)_-^2\,dx.
\end{align}
We denote
\begin{align*}
y_j=\frac{\left|\{u(x,t)\leq k_j\}\cap Q_j\right|}{\left|Q_j\right|}
\end{align*}
and 
$$
 z_j=\left[\mint_0^{\delta \rho^{s p}}\left(\frac{\left|A_-\left(u;0,\rho_j;k_j\right)\right|}{\left|K_j\right|}\right)^{\frac{q}{q-(p-1)}}\,dt\right]^{\frac{q-(p-1)}{q(1+\kappa)}}.
 $$

Via direct computations similar to $I_2$ and $I_3$ in Lemma \ref{lem-3-2}, the first and second terms on the right-hand side of \eqref{3.8} are bounded by 
\begin{align}
\label{3.9}
\gamma 2^{j(N+p)}\left(\frac{M^p}{\rho^{sp}}y_j|Q_j|+\frac{M^p}{\rho^{sp}} z_j^{1+\kappa}\left|Q_j\right|\right),
\end{align}
where we need enforce 
\begin{align*}
\frac{\rho^{\frac{N\kappa}{p-1}}}{\delta^{\frac{1}{q}}}\widetilde{\mathrm{Tail}}(u_-;0,R;0,\delta\rho^{sp})=\left(\frac{\rho}{R}\right)^{\frac{sp}{p-1}}\mathrm{Tail}(u_-;0,R;0,\delta\rho^{sp})\leq M.
\end{align*}
Taking into account $\rho<\rho_j\leq 2\rho$, $M/2<k_j\leq M$, along with the condition \eqref{3.7}, we proceed to deal with the last integral term in \eqref{3.8} as
\begin{align}
\label{3.10}
\int_{K_j}(|u(x,0)|+\widehat{k}_j)^{p-2}(u(x,0)-\widehat{k}_j)_-^2\,dx&\leq\int_{K_j}(|u(x,0)|+\widehat{k}_j)^{p-1}(u(x,0)-k_j)_-\,dx\nonumber\\
&\leq2^p M^p\left|\{u(\cdot,0)<M\}\cap K_j\right|\nonumber\\
& \leq 2^p M^p\left|\{u(\cdot, 0)<M\}\cap K_{2 \rho}\right| \nonumber\\
& \leq  2^p M^p\nu_0\left|K_{2 \rho}\right|\nonumber\\ 
&\leq  2^{N+p} M^p\nu_0\left|K_j\right|=\frac{2^{N+p}\nu_0 M^p }{\delta \rho^{sp}}\left|Q_j\right|,
\end{align}
the parameter $\nu_0$ is to be fixed. 
%Let us focus on the first term on the left-hand side in \eqref{3.8}, it holds that
%\begin{align}
%\label{3.11}
%\operatorname*{ess\sup}_{0<t<\delta \rho^{sp}} \int_{K_{j+1}} (|u|+\widehat{k}_j)^{p-2}(u-\widehat{k}_j)_-^2\,dx\geq\frac{\gamma}{2^{j|p-2|}}\operatorname*{ess\sup}_{0<t<\delta\rho^{sp}} \int_{K_{j+1}}(u-\widehat{k}_j)_-^p(x,t)\,dx.
%\end{align}
Inserting estimates \eqref{3.9}--\eqref{3.10} into \eqref{3.8} infers 
\begin{align}
\label{3.12}
&\quad\operatorname*{ess\sup}_{0<t<\delta \rho^{sp}} \int_{K_{j+1}} (|u|+\widehat{k}_j)^{p-2}(u-\widehat{k}_j)_{-}^2\,dx\nonumber\\
&\quad+\int_0^{\delta \rho^{sp}} \int_{K_{j+1}} \int_{K_{j+1}} \frac{\left|(u-\widehat{k}_j)_-(x,t)-(u-\widehat{k}_j)_-(y,t)\right|^p}{|x-y|^{N+sp}}\,dxdydt\nonumber\\
&\leq \gamma 2^{j(N+p)}\left(\frac{M^p}{\rho^{sp}}y_j|Q_j|+\frac{M^p}{\rho^{sp}} z_j^{1+\kappa}\left|Q_j\right|\right)+ \frac{2^{N+p}\nu_0 M^p}{\delta \rho^{sp}} \left|Q_j\right|.
\end{align}

At this point, we shall consider two separate cases. For the first case, we assume 
\begin{align}
\label{3.13}
\left|\{u(x,t)<k_j\} \cap Q_j\right| \leq \frac{\nu_0}{\delta}\left|Q_j\right| \quad \text{for some }j\in\mathbb{N}_0.
\end{align}
Due to $\rho<\rho_j\leq 2\rho$ and $k_j\geq\frac{1}{2}M$, it tells from \eqref{3.13} that
\begin{align*}
\left|\left\{u<\frac{M}{2}\right\}\cap\left(0, \delta \rho^{sp}\right)+Q_{\rho}(\delta) \right| \leq \frac{\nu_0}{\delta}\left|Q_j\right| \leq 2^N \frac{\nu_0}{\delta}\left|Q_{\rho}(\delta)\right|.
\end{align*}
If $\delta$ is determined, we can select $\nu_0$ to be small enough to satisfy
\begin{align}
\label{3.14}
\nu_0\leq 2^{-N}\bar{\nu}\delta^{l+1},
\end{align}
where $\bar{\nu}$ and $l$ are given in Lemma \ref{lem-3-2} and depends only on $N,p,s,q,\Lambda$. With this choice, it permits us to further utilize Lemma \ref{lem-3-2} to get
\begin{align*}
u \geq \frac{1}{4} M \quad \text { a.e. in }\left(0, \delta \rho^{sp}\right)+Q_{\frac{1}{2} \rho}(\delta).
\end{align*}
Note that the quantitative of $\delta$ is not required yet.
 
For the second case, we suppose the contrary of \eqref{3.13}, that is,
\begin{align}
\label{3.15}
\left|\{u(x,t)\leq k_j\} \cap Q_j\right|> \frac{\nu_0}{\delta}\left|Q_j\right| \quad \text{for some } j\in\mathbb{N}_0.
\end{align}
A combination of \eqref{3.12} and \eqref{3.15} shows that
\begin{align*}
&\quad\operatorname*{ess\sup}_{0<t<\delta \rho^{sp}} \int_{K_{j+1}} (|u|+\widehat{k}_j)^{p-2}(u-\widehat{k}_j)_{-}^2\,dx\nonumber\\
&\quad+\int_0^{\delta \rho^{sp}} \int_{K_{j+1}} \int_{K_{j+1}} \frac{\left|(u-\widehat{k}_j)_-(x,t)-(u-\widehat{k}_j)_-(y,t)\right|^p}{|x-y|^{N+sp}}\,dxdydt\nonumber\\
&\leq \gamma 2^{j(N+p)}\left(\frac{M^p}{\rho^{sp}}y_j|Q_j|+\frac{M^p}{\rho^{sp}} z_j^{1+\kappa}\left|Q_j\right|\right).
\end{align*}
We now perform a similar iteration procedure to that in Lemma \ref{lem-3-2} to yield
\begin{align*}
 y_{j+1} \leq \gamma 2^{j(N+3p+2)} \delta^\ell\left(y_j^{1+\frac{sp}{N+s p}}+z_j^{1+\kappa} y_j^{\frac{sp}{N+sp}}\right)
\end{align*}
and
\begin{align*}
	 z_{j+1} \leq\gamma 2^{j(N+3p+2)} \delta^\ell\left(y_j+z_j^{1+\kappa}\right),
\end{align*}
where $\ell$ is given in \eqref{3.6} and the constant $\gamma$ depends only on $N,p,s,q,\Lambda$.
In view of Lemma \ref{lem-2-4}, we know $Y_j\rightarrow 0$ as $j\rightarrow\infty$ if
\begin{align*}
y_0+z_0^{1+\kappa}\leq \left(2\gamma\delta^\ell\right)^{-\frac{1+\kappa}{\varsigma}}2^{-(N+3p+2)\frac{1+\kappa}{\varsigma^2}},
\end{align*}
with $\varsigma=\min\{\kappa,\frac{sp}{N+sp}\}$. Hence, we can find $\delta\in(0,1)$ depending on $N,p,s,q,\Lambda$ such that 
\begin{align*}
u\geq\frac{1}{4} M \quad \text { a.e. in }\left(0, \delta \rho^{sp}\right)+Q_{\frac{1}{2} \rho}(\delta).
\end{align*}
Since $\delta$ is specified, we can determine $\nu_0$ from \eqref{3.14}, which depends only on $N,p,s,q,\Lambda$. Joining the above two cases, the proof is completed.
\end{proof}

The forthcoming lemma shows the measure propagation of positivity.

\begin{lemma}
	\label{lem-3-4}
Suppose $p>1$ and $u$ is a local weak supersolution to \eqref{1.1} satisfying $u\geq 0$ in $\mathscr{Q}$. Assume further that $u\in L^q_{\rm{loc}}(0,T;L^{sp}_{p-1}(\mathbb{R}^N))$ with $q>p-1$. If there holds
\begin{align}
\label{3.16}
\left|\{u(\cdot, t_0) \geq M \}\cap K_{\rho}(x_0)\right| \geq \alpha\left|K_{\rho}\right|
\end{align}
for constants $M>0$ and $\alpha\in(0,1)$, then one can find $\delta$ and $\varepsilon\in(0,1)$ depending only on $N,p,s,\Lambda$ and $\alpha$, such that either
\begin{align*}
\left(\frac{\rho}{R}\right)^{\frac{sp}{p-1}}\mathrm{Tail}(u_-;x_0,R;t_0,t_0+\delta\rho^{sp})>M,
\end{align*}
or
\begin{align*}
\left|\{u(\cdot, t)\geq\varepsilon M\}\cap K_\rho(x_0)\right| \geq \frac{\alpha}{2}|K_\rho| \quad \text { for all } t \in\left(t_0, t_0+\delta\rho^{s p}\right],
\end{align*}
with imposing $K_\rho(x_0)\times(t_0,t_0+\delta\rho^{sp}]\subset\mathscr{Q}$. Moreover, we can trace the power of dependence as $\varepsilon=\frac{1}{8}\alpha$ and $\delta=\bar{\delta}\alpha^{N+p+1}$ with some $\bar{\delta}\in(0,1)$ depending only on $N,p,s,\Lambda$.
\end{lemma}
\begin{proof}
Assume that $(x_0,t_0)=(0, 0)$. Let us use the Caccioppoli-type inequality \eqref{3.1} in $Q=K_{\rho} \times\left(0,\delta\rho^{sp}\right]$ with $k=M$ to achieve the proof. Choose the cutoff function $\psi(x,t)\equiv\psi(x)\in C_0^\infty \left(K_{\frac{\rho(2-\sigma)}{2}}\right)$ to satisfy $\psi(x)=1$ in $K_{(1-\sigma) \rho}$ and $|\nabla\psi|\leq(\sigma\rho)^{-1}$, where $\sigma$ is to be selected. As a result, we obtain for any $0<t<\delta\rho^{sp}$ that
\begin{align}
\label{3.17}
& \quad \int_{K_{\rho} \times\{t\}} \int_u^M|s|^{p-2}(s-M)_-\,ds \psi^p\,dx \nonumber \\
&\leq \int_{K_{\rho} \times\{0\}} \int_u^M|s|^{p-2}(s-M)_-\,ds \psi^p\,dx \nonumber\\
&\quad+\gamma \iint_Q \int_{K_{\rho}} \max \left\{(u-M)_-^p(x,t), (u-M)_-^p(y,t)\right\} \frac{|\psi(x)-\psi(y)|^p}{|x-y|^{N+sp}}\,dxdydt\nonumber\\
&\quad+\gamma \int_{0}^{\delta\rho^{sp}}\int_{\mathbb{R}^N \backslash K_\rho}\int_{K_\rho} \frac{(u-M)_-^{p-1}(y,t)}{|x-y|^{N+sp}}(u-M)_-(x,t)\psi^p(x,t)\,dxdydt.
\end{align}
Thanks to the condition \eqref{3.16} and $u\geq 0$ in $\mathscr{Q}$, we have
\begin{align*}
\int_{K_{\rho} \times\{0\}} \int_u^M|s|^{p-2}(s-M)_-\,ds \psi^p\,dx \leq(1-\alpha)|K_\rho| \int_0^M|s|^{p-2}(s-M)_-\,ds.
\end{align*}
We continue to address the second integral term on the right-hand side of \eqref{3.17} by using the properties of $\psi(x)$ as
\begin{align*}
\iint_Q \int_{K_{\rho}} \max \left\{(u-M)_-^p(x,t), (u-M)_-^p(y,t)\right\} \frac{|\psi(x)-\psi(y)|^p}{|x-y|^{N+sp}}\,dxdydt\leq \gamma \frac{\delta M^p}{\sigma^p}|K_{\rho}|.
\end{align*}
For the last term, we know
\begin{align*}
\frac{|y|}{|y-x|}\leq 1+\frac{|x|}{|y-x|}\leq \frac{2}{\sigma}
\end{align*}
for all $|x|\leq \frac{\rho(2-\sigma)}{2}$ and $|y|\geq \rho$. This along with H\"{o}lder's inequality gives that
\begin{align*}
&\quad\int_{0}^{\delta\rho^{sp}}\int_{\mathbb{R}^N \backslash K_\rho}\int_{K_\rho} \frac{(u-M)_-^{p-1}(y,t)}{|x-y|^{N+sp}}(u-M)_-(x,t)\psi^p(x,t)\,dxdydt\\
&\leq \frac{\gamma}{\sigma^{N+sp}} \left(\frac{M^p|Q|}{\rho^{sp}} +\int_{0}^{\delta\rho^{sp}}\int_{\mathbb{R}^N \backslash K_R}\int_{K_\rho} \frac{u_-^{p-1}(y,t)}{|x-y|^{N+sp}}(u-M)_-(x,t)\,dxdydt \right)\\
&\leq  \frac{\gamma}{\sigma^{N+sp}} \left(\frac{M^p|Q|}{\rho^{sp}}+(\delta\rho^{sp})^{\frac{q-(p-1)}{q}} M|K_\rho|[\widetilde{\mathrm{Tail}}(u_-;0,R;0,\delta\rho^{sp})]^{p-1}\right) \\
&\leq \gamma \frac{\delta M^p}{\sigma^{N+sp}}|K_\rho|,
\end{align*}
if imposing
\begin{align*}
\frac{\rho^{\frac{N\kappa}{p-1}}}{\delta^{\frac{1}{q}}}\widetilde{\mathrm{Tail}}(u_-;0,R;0,\delta\rho^{sp})=\left(\frac{\rho}{R}\right)^{\frac{sp}{p-1}}\mathrm{Tail}(u_-;0,R;0,\delta\rho^{sp})\leq M.
\end{align*}
Collecting these estimates together leads to
\begin{align*}
&\quad \int_{K_{\rho} \times\{t\}} \int_u^M|s|^{p-2}(s-M)_-\,ds \psi^p\,dx \\
& \leq 
(1-\alpha)|K_\rho| \int_0^M|s|^{p-2}(s-M)_-\,ds+\gamma \frac{\delta M^p}{\sigma^{N+p}}|K_{\rho}|.
\end{align*}
Hereafter, the conclusion can be derived following the same steps as in the proof of Lemma 4.4 in \cite{P24}. We omit the details for short.
%For representing the dependence of $\delta$ and $\varepsilon$ more clearly, we give the details here.
\end{proof}

What follows is a measure shrinking lemma that investigates the measure density information in the cylinder from the measure condition in a ball.

\begin{lemma}
\label{lem-3-5}
Suppose $p>1$ and $u$ is a local weak supersolution to \eqref{1.1} satisfying $u\geq 0$ in $\mathscr{Q}$. We assume additionally that $u\in L^q_{\rm{loc}}(0,T;L^{sp}_{p-1}(\mathbb{R}^N))$ with $q>p-1$. Let parameters $M>0$, $\delta,\sigma\in(0,\frac{1}{2})$ and the cylinder $Q_{2\rho}(\delta)$ be included in $ \mathscr{Q}$. If we have
\begin{align*}
\left|\{u(\cdot,t)>M  \}\cap K_{\rho}(x_0)\right|  \geq \alpha\left|K_{\rho}\right| \quad \text { for all } t \in\left(t_0-\delta \rho^{sp}, t_0\right],
\end{align*}
then whenever
\begin{align*}
\left(\frac{\rho}{R}\right)^{\frac{sp}{p-1}}\mathrm{Tail}(u_-;x_0,R;t_0-\delta\rho^{sp},t_0)\leq\sigma M,
\end{align*}
there exists a constant $\gamma(N,p,s,\Lambda)>0$ such that
\begin{align*}
\left|\{u \leq \sigma M \}\cap Q_{\rho}(\delta)\right| \leq \gamma \frac{\sigma^{p-1}}{\delta \alpha}\left|Q_{\rho}(\delta)\right|.
\end{align*}

\end{lemma}
\begin{proof}
Assume that $(x_0,t_0)=(0,0)$. We introduce $\psi(x,t)\equiv\psi(x)\in C_0^\infty(K_{\frac{3}{2}\rho})$, which  is a cutoff function fulfilling $\psi(x)=1$ in $K_\rho$ and $|\nabla\psi|\leq \rho^{-1}$. In such a way, Lemma \ref{lem-3-1} in the cylinder $K_{2 \rho} \times(-\delta \rho^{sp},0]$ with $k=\sigma M$ results in
\begin{align}
\label{3.18}
&\quad\iint_{Q_\rho(\delta)}  (u-k)_-(y,t)\int_{K_{2\rho}} \frac{(u-k)_+^{p-1}(x, t)}{|x-y|^{N+sp}}\,dxdydt\nonumber\\
&\leq  \gamma \int_{-\delta \rho^{sp}}^0 \int_{K_{2 \rho}} \int_{K_{2 \rho}} \max \left\{(u-k)_-^p(x,t),(u-k)_-^p(y,t)\right\} \frac{|\psi(x)-\psi(y)|^p}{|x-y|^{N+s p}}\,dxdydt \nonumber\\
&\quad+\gamma \int_{-\delta \rho^{sp}}^0\int_{\mathbb{R}^N \backslash K_{2\rho}} \int_{K_{2 \rho}} \frac{(u-k)_-^{p-1}(y,t)}{|x-y|^{N+sp}}(u-k)_-(x,t)\psi^p(x,t)\,dxdt\nonumber\\
&\quad +\int_{K_{2 \rho}}(|u(x,-\delta\rho^{sp})|+k)^{p-2}(u(x,-\delta\rho^{sp})-k)_-^2\,dx.
\end{align}
Proceeding similarly as in Lemma \ref{lem-3-4}, it is not hard to examine the upper bound of the terms on the right-hand side of \eqref{3.18} is
\begin{align*}
\gamma \frac{(\sigma M)^p}{\delta \rho^{s p}}|Q_{\rho}(\delta)|,
\end{align*} 
if enforcing 
\begin{align*}
\frac{\rho^{\frac{N\kappa}{p-1}}}{\delta^{\frac{1}{q}}}\widetilde{\mathrm{Tail}}(u_-;0,R;-\delta\rho^{sp},0)=\left(\frac{\rho}{R}\right)^{\frac{sp}{p-1}}\mathrm{Tail}(u_-;0,R;-\delta\rho^{sp},0)\leq\sigma M.
\end{align*}
When it comes to the left-hand side of \eqref{3.18}, we estimate 
\begin{align*}
&\quad\iint_{Q_\rho(\delta)}  (u-k)_-(y,t)\chi_{\{u(y,t)\leq\frac{1}{2}\sigma M\}}\int_{K_{2\rho}} \frac{(u-k)_+^{p-1}(x,t)\chi_{\{u(x,t)\geq M\}}}{|x-y|^{N+sp}}\,dxdydt\nonumber\\
 &\geq \frac{1}{2} \sigma M\left|\left\{u \leq \frac{1}{2} \sigma M\right\} \cap Q_{\rho}(\delta)\right|\left[\frac{\left(\frac{1}{2} M\right)^{p-1} \alpha\left|K_{\rho}\right|}{(4 \rho)^{N+s p}}\right] \\
 &=2^{-(2 N+p+2 s p)} \frac{M^p \alpha \sigma}{\rho^{s p}}\left|\left\{u \leq \frac{1}{2} \sigma M\right\} \cap Q_{\rho}(\delta)\right| . 
\end{align*}
We immediately reach the conclusion by gathering these estimates.
\end{proof}

Finally, we give a lemma that removes the measure theoretical condition at each time level required in the previous lemma and is expressed in terms of local integrals.

\begin{lemma}
\label{lem-3-6}
Let $p>1$ and $u$ be a local weak supersolution to \eqref{1.1} that satisfies $u \geq 0$ in $\mathscr{Q}$. We assume additionally that $u\in L^q_{\rm{loc}}(0,T;L^{sp}_{p-1}(\mathbb{R}^N))$ with $q>p-1$. Suppose constant $M>0$ and the cylinder $(x_0,t_0)+Q_{2\rho}$ is included in $\mathscr{Q}$. Then we can find a constant $\gamma(N,p,s,\Lambda)>1$, such that there holds either
\begin{align*}
\left(\frac{\rho}{R}\right)^{\frac{sp}{p-1}}\mathrm{Tail}(u_-;x_0,R;t_0-\rho^{sp},t_0)>M,
\end{align*}
or
\begin{align*}
\mathop{\mathrm{ess}\inf}_{t_0-\rho^{sp}\leq t\leq t_0}\left|\{u(\cdot, t)\leq M\} \cap K_{\rho}(x_0)\right| \leq \frac{\gamma M^{p-1}}{[u^{p-1}]_{Q_{\rho}}}|K_{\rho}|,
\end{align*}
where $[\cdot]_{Q_{\rho}}$ denotes the integral average on $\left(x_0, t_0\right)+Q_{\rho}$.
\end{lemma}
\begin{proof}
Let $(x_0,t_0)=(0,0)$. We plan to employ the energy estimate \eqref{3.1} in $K_{2 \rho} \times\left(-\rho^{sp}, 0\right]$ with $k=M$. Let $\psi(x,t)\equiv\psi(x)\in C_0^\infty(K_{3\rho/2})$ be a cutoff function such that $\psi=1$ in $K_\rho$ and $|\nabla \psi|\leq 4 / \rho$. Then it turns into
\begin{align*}
&\quad \iint_{Q_\rho}  (u-M)_-(y,t)\int_{K_{2\rho}} \frac{(u-M)_+^{p-1}(x, t)}{|x-y|^{N+sp}}\,dxdydt\nonumber\\
&\leq  \gamma \int_{-\rho^{sp}}^0 \int_{K_{2 \rho}} \int_{K_{2 \rho}} \max \left\{(u-M)_-^p(x,t),(u-M)_-^p(y,t)\right\} \frac{|\psi(x)-\psi(y)|^p}{|x-y|^{N+s p}}\,dxdydt \nonumber\\
&\quad+\gamma\int_{-\rho^{sp}}^0 \int_{\mathbb{R}^N \backslash K_{2\rho}}\int_{K_{2 \rho}} \frac{(u-M)_-^{p-1}(y,t)}{|x-y|^{N+sp}}(u-M)_-(x,t)\psi^p(x,t)\,dxdt\nonumber\\
&\quad+\int_{K_{2 \rho}}(|u(x,-\rho^{sp})|+M)^{p-2}(u(x,-\rho^{sp})-M)_-^2.
\end{align*}
Analogous to Lemma \ref{lem-3-4}, after enforcing
\begin{align*}
\rho^{\frac{N\kappa}{p-1}}\widetilde{\mathrm{Tail}}(u_-;0,R;-\rho^{sp},0)=\left(\frac{\rho}{R}\right)^{\frac{sp}{p-1}}\mathrm{Tail}(u_-;0,R;-\rho^{sp},0)\leq M,
\end{align*}
we can obtain
\begin{align}
\label{3.19}
\iint_{Q_\rho}  (u-M)_-(y,t)\int_{K_{2\rho}} \frac{(u-M)_+^{p-1}(x, t)}{|x-y|^{N+sp}}\,dxdydt\leq& \gamma M^p |K_{2\rho}|.
\end{align}

On the other hand, we further deduce that
\begin{align}
\label{3.20}
& \quad \iint_{Q_\rho}  (u-M)_-(y,t)\int_{K_{2\rho}} \frac{(u-M)_+^{p-1}(x, t)}{|x-y|^{N+sp}}\,dxdydt\nonumber\\
&\geq\iint_{Q_\rho}  (u-M)_-(y,t)\int_{K_{\rho}} \frac{(u-M)_+^{p-1}(x, t)}{(2\rho)^{N+sp}}\,dxdydt\nonumber\\
&\geq \underset{-\rho^{sp}\leq t\leq 0}{\operatorname{ess} \inf } \int_{K_{\rho}} (u-M)_-(y,t)\iint_{Q_{\rho}} \frac{(u-M)^{p-1}_+(x,t)}{(2\rho)^{N+sp}}\,dxdtdy\nonumber\\
&\geq \gamma\underset{-\rho^{sp}\leq t\leq 0}{\operatorname{ess} \inf } \int_{K_{\rho}} (u-M)_-(y,t)\iint_{Q_{\rho}} \frac{u^{p-1}-M^{p-1}}{(2\rho)^{N+sp}}\,dxdtdy\nonumber\\
&\geq \gamma\underset{-\rho^{sp}\leq t\leq 0}{\operatorname{ess} \inf } \int_{K_{\rho}} (u-M)_-(y,t)\left(\mint\!\!\mint_{Q_{\rho}} u^{p-1} \,dxdt- M^{p-1}\right)\,dy\nonumber \\
&\geq \gamma \underset{-\rho^{sp}\leq t\leq 0}{\operatorname{ess} \inf } \int_{K_{\rho}} (u-M)_{-}(y,t)\,dy\cdot[u^{p-1}]_{Q_{\rho}}-\gamma M^p\left|K_{\rho}\right|\nonumber\\
&\geq \frac{\gamma M}{2}  \underset{-\rho^{sp}\leq t\leq 0 }{\operatorname{ess} \inf }\left|\left\{u(\cdot, t) \leq \frac{1}{2} M\right\}\cap K_{\rho}\right|\cdot[u^{p-1}]_{Q_{\rho}}-\gamma M^p\left|K_{\rho}\right|,
\end{align}
where $\gamma=\gamma(N,s,p)>0$. \eqref{3.19} in conjunction with \eqref{3.20} implies that
\begin{align*}
\underset{-\rho^{sp}\leq t\leq 0}{\operatorname{ess} \inf }\left|\left\{u(\cdot, t) \leq \frac{1}{2} M\right\} \cap K_{\rho}\right| \leq \frac{\gamma M^{p-1}}{[u^{p-1}]_{Q_{\rho}}}\left|K_{\rho}\right| .
\end{align*}
In summary, we arrive at the claim after adjusting some constants.
\end{proof}

\section{Expansion of positivity}
\label{sec4}
This section concerns the expansion of positivity, which is a key ingredient in establishing weak Harnack estimates.
\begin{proposition}
\label{pro-4-1}
Let $p>1$ and $u$ be a local weak supersolution to \eqref{1.1} that satisfies $u \geq 0$ in $\mathscr{Q}$. Assume additionally that $u\in L^q_{\rm{loc}}(0,T;L^{sp}_{p-1}(\mathbb{R}^N))$ with $q>p-1$. Suppose constants $M>0$ and $\alpha\in(0,1]$. If we have
\begin{align}
\label{4.1}
\left|\left\{u(\cdot, t_0) \geq M\right\} \cap K_{\rho}(x_0)\right| \geq \alpha|K_{\rho}|,
\end{align}
then one can find a constant $\eta\in(0,1)$ depending only on $N,p,s,q,\Lambda$ and $\alpha$ such that either
\begin{align*}
\left(\frac{\rho}{R}\right)^{\frac{sp}{p-1}}\mathrm{Tail}\left(u_-;x_0,R;t_0,t_0+4(8 \rho)^{sp}\right)>\eta M,
\end{align*}
or
\begin{align*}
u \geq \eta M \quad \text { a.e. in } K_{2 \rho}\left(x_0\right) \times\left(t_0+\frac{1}{2}(8 \rho)^{sp}, t_0+2(8 \rho)^{sp}\right],
\end{align*}
with imposing $K_{4 \rho}(x_0) \times\left(t_0, t_0+4(8 \rho)^{sp}\right] \subset \mathscr{Q}$. Moreover, we can express $\eta$ as $\eta=\bar{\eta} \alpha^\vartheta$ with some $\bar{\eta}\in(0,1)$ and $\vartheta>1$, both of which depend only on $N,p,s,q,\Lambda$.
\end{proposition}
\begin{proof}
For simplicity, let $(x_0,t_0)=(0,0)$. We rewrite \eqref{4.1} as
\begin{align*}
\left|\left\{u(\cdot,0) \geq M\right\} \cap K_{4\rho}\right| \geq 4^{-N} \alpha|K_{\rho}|. 
\end{align*}
According to Lemma \ref{lem-3-4}, if enforcing
\begin{align}
\label{4.2}
\widetilde{\gamma}\left(\frac{\rho}{R}\right)^{\frac{sp}{p-1}}\mathrm{Tail}(u_-;0,R;0,\delta(4\rho)^{sp})\leq M,
\end{align}
with a constant $\widetilde{\gamma}(s,p)>1$, there exist constants $\varepsilon=\frac{1}{8}\alpha$ and $\delta=\bar{\delta}\alpha^{N+p+1}$ with $\bar{\delta}\in(0,1)$ in terms of $N,p,s,\Lambda$ such that
\begin{align*}
%\label{4.3}
\left|\{u(\cdot, t)\geq \varepsilon M\} \cap K_{4 \rho}\right| \geq \frac{\alpha}{2} 4^{-N}\left|K_{4 \rho}\right| \quad \text { for all } t \in\left(0, \delta(4 \rho)^{sp}\right].
\end{align*}
Observe that the following holds true
\begin{align*}
(0,\bar{t})+Q_{4 \rho}\left(\frac{1}{2} \delta\right) \subset K_{4 \rho} \times\left(0,\delta(4 \rho)^{sp}\right]
\end{align*}
for any time
\begin{align}
\label{4.4}
\bar{t} \in\left(\frac{1}{2} \delta(4 \rho)^{sp}, \delta(4 \rho)^{sp}\right].
\end{align}
Consequently, we can utilize Lemma \ref{lem-3-5} in the cylinder $(0, \bar{t})+Q_{4 \rho}\left(\frac{1}{2} \delta\right)$ whenever $\bar{t}$ is in the time interval \eqref{4.4} with $M$ and $\alpha$ replaced by $\varepsilon M$ and $\frac{1}{2}4^{-N}\alpha$. At this stage, we shall select $\sigma$ from Lemma \ref{lem-3-5} such that 
\begin{align*}
\gamma \frac{\sigma^{p-1}}{\delta \alpha} \leq \nu, \quad \text { i.e., } \quad \sigma=\left(\frac{\nu \delta \alpha}{\gamma}\right)^{\frac{1}{p-1}}=\left(\frac{\bar{\nu} \bar{\delta}^{l+1} \alpha^{(N+p+1)(l+1)+1}}{\gamma}\right)^{\frac{1}{p-1}},
\end{align*}
where the constant $\nu=\bar{\nu}\delta^l$ and $\delta=\bar{\delta}\alpha^{N+p+1}$ are fixed in Lemma \ref{lem-3-2} and Lemma \ref{lem-3-4}, respectively. Then we arrive at
\begin{align*}
\left|\{u\leq\sigma\varepsilon M\}\cap(0,\bar{t})+Q_{4\rho}\left(\frac{1}{2}\delta\right)\right|\leq \nu \left|Q_{4\rho}\left(\frac{1}{2}\delta\right)\right|,
\end{align*}
whenever there holds
\begin{align}
\label{4.5}
\widetilde{\gamma}\left(\frac{\rho}{R}\right)^{\frac{sp}{p-1}}\mathrm{Tail}(u_-;0,R;0,\delta(4\rho)^{sp})\leq \sigma\varepsilon M.
\end{align}

Subsequently, let us evoke Lemma \ref{lem-3-2} in the cylinder $(0, \bar{t})+Q_{4 \rho}\left(\frac{1}{2} \delta\right)$ with $M$ replaced by $\sigma\varepsilon M$. Owing to the arbitrariness of $\bar{t}$ in \eqref{4.4}, we get
\begin{align}
\label{4.6}
u \geq \frac{1}{4} \sigma \varepsilon M \quad \text { a.e. in } K_{2 \rho} \times\left(\frac{1}{2} \delta(4 \rho)^{sp}, \delta(4 \rho)^{sp}\right],
\end{align}
where we enforced \eqref{4.5} again.

With the pointwise estimate \eqref{4.6} at hand, we can repeat the preceding process with $\alpha=1$. Let $\eta=\frac{1}{4} \sigma \varepsilon$, it infers that
\begin{align*}
u \geq \bar{\eta} \eta M \quad \text { a.e. in } K_{2 \rho} \times\left(\frac{1}{2} \delta(4 \rho)^{sp}+\frac{1}{2} \bar{\delta}(8 \rho)^{sp}, \delta(4 \rho)^{sp}+\bar{\delta}(8 \rho)^{sp}\right],
\end{align*}
with some constants $\bar{\eta}, \bar{\delta} \in\left(0, \frac{1}{2}\right)$ depending on $N,p,s,q,\Lambda$ and independent of $\alpha$. Here, we need to impose
\begin{align*}
\widetilde{\gamma}\left(\frac{\rho}{R}\right)^{\frac{sp}{p-1}}\mathrm{Tail}\left(u_-;0,R;0,\delta(4 \rho)^{sp}+\bar{\delta}(8 \rho)^{sp}\right)\leq \bar{\eta}\eta M.
\end{align*}
Running the procedure repeatedly, we have for $n\in\mathbb{N}$,
\begin{align}
\label{4.7}
u \geq \bar{\eta}^n \eta M \quad \text { a.e. in } K_{2 \rho} \times\left(\frac{1}{2} \delta(4 \rho)^{sp}+\frac{1}{2} \bar{\delta}(8 \rho)^{sp}, \delta(4 \rho)^{sp}+n \bar{\delta}(8 \rho)^{sp}\right],
\end{align}
provided that 
\begin{align}
\label{4.8}
\widetilde{\gamma}\left(\frac{\rho}{R}\right)^{\frac{sp}{p-1}}\mathrm{Tail}\left(u_-;0,R;0,\delta(4 \rho)^{sp}+n\bar{\delta}(8 \rho)^{sp}\right)\leq \bar{\eta}^n\eta M.
\end{align}
Setting $n=\lceil 2/\bar{\delta}\rceil$, it gives from \eqref{4.7} and $\delta=\bar{\delta}\alpha^{N+p+1}$ that
\begin{align*}
u \geq \bar{\eta}^{\lceil 2 / \bar{\delta}\rceil} \eta M \quad \text { a.e. in } K_{2 \rho} \times\left(\bar{\delta}(8 \rho)^{sp}, (8 \rho)^{sp}\right],
\end{align*}
if \eqref{4.2}, \eqref{4.5} and \eqref{4.8} hold. Considering these requirements together, we only need to ensure
\begin{align*}
\widetilde{\gamma}\left(\frac{\rho}{R}\right)^{\frac{sp}{p-1}}\mathrm{Tail}\left(u_-;0,R;0,4(8 \rho)^{sp}\right)\leq \bar{\eta}^{\lceil 2 / \bar{\delta}\rceil} \eta M,
\end{align*}
with some $\widetilde{\gamma}(s,p)>1$. After redefining some constants, we get the desired result.
\end{proof}

\section{Proof of Theorems \ref{thm-1-2} and  \ref{thm-1-3}}
\label{sec5}
In this section, we are ready to prove the main results on weak Harnack estimates for weak supersolutions based on the expansion of positivity.

\begin{proof}[Proof of Theorem \ref{thm-1-2}]
We may assume $(x_0,t_0)=(0,0)$. Denote
\begin{align*}
\mu=\left(\frac{\rho}{R}\right)^{\frac{sp}{p-1}}\mathrm{Tail}\left(u_-;0,R;0,4(8 \rho)^{sp}\right)+\underset{K_{2 \rho} \times\left(\frac{1}{2}(8 \rho)^{sp}, 2(8 \rho)^{sp}\right]}{\operatorname{ess} \inf } u .
\end{align*}
Letting $\beta>0$ to be chosen, we calculate 
\begin{align}
\label{5.1}
\int_{K_{\rho}} u^{\beta}(\cdot, 0)\,dx & =\beta \int_0^{\infty}\left|\{u(\cdot, 0)>M\} \cap K_{\rho}\right| M^{\beta-1}\,dM \nonumber\\
& \leq \beta \int_{\mu}^{\infty}\left|\{u(\cdot, 0)>M\} \cap K_{\rho}\right| M^{\beta-1}\,dM +\mu^\beta\left|K_{\rho}\right|.
\end{align}
From Proposition \ref{pro-4-1}, there exist constants $\vartheta>1$ and $\eta\in(0,1)$  depending on $N,p,s,q,\Lambda$ such that  
\begin{align*}
\mu \geq \eta M\left(\frac{\left|\{u(\cdot, 0)>M\} \cap K_{\rho}\right|}{\left|K_{\rho}\right|}\right)^\vartheta .
\end{align*}
Thereby, we treat the first part on the right-hand side of \eqref{5.1} as
\begin{align*}
\beta \int_{\mu}^{\infty}\left|\{u(\cdot, 0)>M\} \cap K_{\rho}\right| M^{\beta-1}\,dM \leq \frac{\beta \mu^{\frac{1}{\vartheta}}}{\eta^{\frac{1}{\vartheta}}}\left|K_{\rho}\right| \int_{\mu}^{\infty} M^{\beta-\frac{1}{\vartheta}-1} \,dM.
\end{align*}
Here, we shall let $\beta<\frac{1}{\vartheta}$ to deduce
\begin{align*}
\int_{K_{\rho}} u^{\beta}(\cdot, 0)\,dx\leq \left[1+\frac{\beta}{\left(\beta-\frac{1}{\vartheta}\right) \eta^{\frac{1}{\vartheta}}}\right] \mu^{\beta}\left|K_{\rho}\right|,
\end{align*}
as intended.
\end{proof}

\medskip

\begin{proof}[Proof of Theorem \ref{thm-1-3}]
Let us utilize Lemma \ref{lem-3-6} and select constant $M>0$ such that
\begin{align}
\label{5.2}
\underset{-(2 \rho)^{sp}\leq t\leq 0}{\operatorname{ess} \inf }\left|\{u(\cdot,t) \leq M\} \cap K_{2 \rho}\right| \leq \frac{\gamma M^{p-1}}{[u^{p-1}]_{Q_{2\rho}}}\left|K_{2 \rho}\right| \leq \nu_0\left|K_{2 \rho}\right|,
\end{align}
where $\nu_0$ is determined as in Lemma \ref{lem-3-3}. In other words, \eqref{5.2} amounts to requiring
\begin{align}
\label{5.3}
M=\left(\frac{\nu_0[u^{p-1}]_{Q_{2\rho}}}{\gamma}\right)^\frac{1}{p-1}.
\end{align}
Set $\bar{t} \in\left[-(2 \rho)^{sp}, 0\right]$ as the point where the infimum in \eqref{5.2} is attained. By means of Lemma \ref{lem-3-3}, there exists some $\delta \in\left(0, \frac{1}{4}\right)$ depending only on $N,p,s,q,\Lambda$ such that
\begin{align}
\label{5.4}
u \geq \frac{1}{4} M \quad \text { a.e. in } K_{\frac{1}{2} \rho} \times\left(\bar{t}+\frac{3}{4} \delta \rho^{sp}, \bar{t}+\delta \rho^{sp}\right],
\end{align}
if $K_{2 \rho} \times\left[-(2 \rho)^{sp}, \delta\rho^{sp}\right] \subset \mathscr{Q}$ and
\begin{align}
\label{5.5}
\left(\frac{\rho}{R}\right)^{\frac{sp}{p-1}}\mathrm{Tail}\left(u_-;0,R;-(2 \rho)^{sp}, \delta\rho^{sp}\right) \leq M .
\end{align}
Afterward, the estimate \eqref{5.4} allows us to exploit Proposition \ref{pro-4-1} with $\alpha=1$. As a consequence, we can find a constant $\eta$ depending only on $N,p,s,q,\Lambda$ such that
\begin{align}
\label{5.6}
u \geq \eta M \quad \text { a.e. in } K_{\rho} \times\left(\bar{t}+\frac{3}{4} \delta \rho^{sp}+\frac{1}{2}(4 \rho)^{sp}, \bar{t}+\delta \rho^{sp}+2(4 \rho)^{sp}\right],
\end{align}
if imposing $K_{2 \rho} \times\left(-(2\rho)^{sp}, 4(4 \rho)^{sp}\right] \subset \mathscr{Q}$ and
\begin{align}
\label{5.7}
\left(\frac{\rho}{R}\right)^{\frac{sp}{p-1}}\mathrm{Tail}\left(u_-;0,R;-(2 \rho)^{sp}, 4(4 \rho)^{sp}\right) \leq \eta M .
\end{align}
Indeed, we only need to enforce \eqref{5.7} since $0\leq\eta\leq 1$, and thus \eqref{5.5} holds directly. Finally, it follows from \eqref{5.6} that
\begin{align*}
u \geq \eta M \quad \text { a.e. in } K_{\rho} \times\left(\frac{3}{4}(4 \rho)^{sp},(4 \rho)^{sp}\right],
\end{align*}
whenever $\bar{t} \in\left[-(2\rho)^{sp}, 0\right]$. Theorem \ref{thm-1-3} follows readily by considering \eqref{5.3} and redefining $(\eta \nu_0 /\gamma)^{\frac{1}{p-1}}$ as $\eta$.
\end{proof}

\bigskip

\section*{Acknowledgments}
This work was supported by the National Natural Science Foundation of China (No. 12471128). 

\section*{Declarations}
\subsection*{Conflict of interest} The authors declare that there is no conflict of interest. We also declare that this
manuscript has no associated data.

\subsection*{Data availability} Data sharing is not applicable to this article as no datasets were generated or analysed
during the current study.


\begin{thebibliography}{[a]}
	
\bibitem{AF89} E. Acerbi and N. Fusco, Regularity for minimizers of nonquadratic functionals: the case $1<p<2$, J. Math. Anal. Appl. 140 (1) (1989) 115--135.
	
\bibitem{ASD08} R. Alonso, M. Santillana and C. Dawson, On the diffusive wave approximation of the shallow water equations, European J. Appl. Math. 19 (5) (2008) 575--606. 	

\bibitem{BGK22} A. Banerjee, P. Garain and J. Kinnunen, Some local properties of subsolutions and supersolutions for a doubly nonlinear nonlocal $p$-Laplace equation, Ann. Mat. Pura Appl. 201 (4) (2022) 1717--1751.

\bibitem{BGK23} A. Banerjee, P. Garain and J. Kinnunen, Lower semicontinuity and pointwise behavior of supersolutions for some doubly nonlinear nonlocal parabolic $p$-Laplace equations, Commun. Contemp. Math. 25 (8) (2023) 23pp.

\bibitem{B24} F. B\"{a}uerlein, Weak Harnack inequality for doubly non-linear equations of slow diffusion type, J. Math. Anal. Appl. 539 (2) (2024) 40pp.
	
\bibitem{BDG23} V. B\"{o}gelein, F. Duzaar, U. Gianazza, N. Liao and C. Scheven, H\"{o}lder continuity of the gradient of solutions to doubly non-linear parabolic equations, arXiv:2305.08539v1.

\bibitem{BDKS20} V. B\"{o}gelein, F. Duzaar, J. Kinnunen and C. Scheven,  Higher integrability for doubly nonlinear parabolic systems, J. Math. Pures Appl. 143 (9) (2020) 31--72.	

\bibitem{BDL21}V. B\"{o}gelein, F. Duzaar and N. Liao, On the H\"{o}lder regularity of signed solutions to a doubly nonlinear equation, J. Funct. Anal. 281 (9) (2021) 58pp.

\bibitem{BHSS21} V. B\"{o}gelein, A. Heran, L. Sch\"{a}tzler and T. Singer, Harnack's inequality for doubly nonlinear equations of slow diffusion type, Calc. Var. Partial Differential Equations 60 (6) (2021) 35pp.

\bibitem{BS24} V. B\"{o}gelein and M. Strunk, A comparison principle for doubly nonlinear parabolic partial differential equations, Ann. Mat. Pura Appl. 203 (2) (2024) 779--804.


\bibitem{BK24} S. Byun and K. Kim, A H\"{o}lder estimate with an optimal tail for nonlocal parabolic $p$-Laplace equations, Ann. Mat. Pura Appl. 203 (1) (2024) 109--147.

\bibitem{D93} E. DiBenedetto, Degenerate parabolic equations, Springer-Verlag, New York, 1993.		

\bibitem{DZZ21} M. Ding, C. Zhang and S. Zhou, Local boundedness and H\"{o}lder continuity for the parabolic fractional $p$-Laplace equations, Calc. Var. Partial Differential Equations 60 (1) (2021) 45pp.

\bibitem{FK13} M. Felsinger and M. Kassmann, Local regularity for parabolic nonlocal operators, Comm. Partial Differential Equations 38 (9) (2013) 1539--1573.

\bibitem{GK23} P. Garain and J. Kinnunen, Weak Harnack inequality for a mixed local and nonlocal parabolic equation, J. Differential Equations 360 (2023) 373--406.

\bibitem{GV06} U. Gianazza and V. Vespri, A Harnack inequality for solutions of doubly nonlinear parabolic equations, J. Appl. Funct. Anal. 1 (3) (2006) 271--284.

\bibitem{GM86} M. Giaquinta and G. Modica, Remarks on the regularity of the minimizers of certain degenerate functionals, Manuscripta Math. 57 (1) (1986) 55--99.	

\bibitem{KS14} M. Kassmann and R. Schwab, Regularity results for nonlocal parabolic equations, Riv. Mat. Univ. Parma (N.S.) 5 (1) (2014) 183--212.	

\bibitem{KW22a} M. Kassmann and M. Weidner, Nonlocal operators related to nonsymmetric forms I: H\"{o}lder estimates, arXiv:2203.07418v1.

\bibitem{KW22b} M. Kassmann and M. Weidner, Nonlocal operators related to nonsymmetric forms II: Harnack inequalities, Anal. PDE,  to appear.

\bibitem{KW23} M. Kassmann and M. Weidner, The parabolic Harnack inequality for nonlocal equations, Duke Math. J.,  to appear.

\bibitem{K19} Y. Kim, Nonlocal Harnack inequalities for nonlocal heat equations,
J. Differential Equations 267 (11) (2019) 6691--6757.

\bibitem{KK07} J. Kinnunen and T. Kuusi, Local behaviour of solutions to doubly nonlinear parabolic equations, Math. Ann. 337 (3) (2007) 705--728.

\bibitem{KSLU12} T. Kuusi, R. Laleoglu, J. Siljander and J. Urbano, H\"{o}lder continuity for Trudinger's equation in measure spaces, Calc. Var. Partial Differential Equations 45 (1-2) (2012) 193--229.

\bibitem{KSU12} T. Kuusi, J. Siljander and J. Urbano, Local H\"{o}lder continuity for doubly nonlinear parabolic equations, Indiana Univ. Math. J. 61 (1) (2012) 399--430.

\bibitem{LM18} G. Leugering and G. Mophou, Instantaneous optimal control of friction dominated flow in a gas-network, Shape optimization, Homogenization and optimal control, Springer, Cham, 2018.

\bibitem{L18} Q. Li, Weak Harnack estimates for supersolutions to doubly degenerate parabolic equations, Nonlinear Anal. 170 (2018) 88--122.

\bibitem{Li24} N. Liao, On the modulus of continuity of solutions to nonlocal parabolic equations, J. Lond. Math. Soc. 110 (3) (2024) 30pp.

\bibitem{L24} N. Liao, Nonlocal weak Harnack estimates, arXiv:2402.11986v1.

\bibitem{LL22} E. Lindgren and P. Lindqvist, On a comparison principle for
 Trudinger's equation, Adv. Calc. Var. 15 (3) (2022) 401--415.

\bibitem{M76} M. Mahaffy, A three-dimensional numerical model of ice sheets: tests on the barnes ice cap, northwest territories, J. Geophys. Res. 81 (6) (1976) 1059--1066.

\bibitem{N23} K. Nakamura, Harnack's estimate for a mixed local-nonlocal doubly nonlinear parabolic equation, Calc. Var. Partial Differential Equations 62 (2) (2023) 45pp.

\bibitem{P24} H. Prasad, On the weak Harnack estimate for nonlocal equations, Calc. Var. Partial Differential Equations 63 (3) (2024) 19pp.

\bibitem{RSZ24} V. R\u{a}dulescu, B. Shang and C. Zhang, Harnack inequality for doubly nonlinear mixed local and nonlocal parabolic equations, arXiv:2406.03889v1.

\bibitem{S10} J. Siljander, Boundedness of the gradient for a doubly nonlinear parabolic equation, J.~Math. Anal. Appl. 371 (1) (2010) 158--167.

\bibitem{S19} M. Str\"{o}mqvist, Harnack's inequality for parabolic nonlocal equations, Ann. Inst. H. Poincar\'{e} Anal. Non Lin\'{e}aire 36 (6) (2019) 1709--1745.

\bibitem{T68} N. Trudinger, Pointwise estimates and quasilinear parabolic equations, Comm. Pure Appl. Math. 21 (7) (1968) 205--226.

\bibitem{V94} V. Vespri, Harnack type inequalities for solutions of certain doubly nonlinear parabolic equations, J.~Math. Anal. Appl. 181 (1) (1994) 104--131. 

\bibitem{VV22} V. Vespri and M. Vestberg, An extensive study of the regularity of solutions to doubly singular equations, Adv. Calc. Var. 15 (3) (2022) 435--473.


\end{thebibliography}
\end{document}